\documentclass[11pt,a4paper]{article}
\usepackage{a4wide}

\usepackage{amsmath}
\usepackage{amssymb}
\usepackage{amsthm}
\usepackage{mathtools}
\usepackage[ocgcolorlinks]{hyperref}
\usepackage{cleveref}
\usepackage{url}
\usepackage{epsfig}
\usepackage{psfrag}

\usepackage{mathptmx}
\usepackage[scaled]{helvet}   
\usepackage{courier}

\usepackage[T1]{fontenc}
\usepackage[utf8]{inputenc}

\usepackage{wrapfig}
\usepackage{algorithm}
\usepackage{algorithmic}
\usepackage{paralist}
\usepackage{xcolor}

\colorlet{DarkRed}{red!60!black}
\colorlet{DarkGreen}{green!50!black}
\colorlet{DarkBlue}{blue!60!black}

\hypersetup{
	linkcolor = DarkRed,
	citecolor = DarkGreen,
	urlcolor = DarkBlue,
	bookmarksnumbered = true,
	linktocpage = true
}

\newtheorem{theorem}{Theorem}
\newtheorem{lemma}{Lemma}
\newtheorem{corollary}{Corollary}
\newtheorem{fact}{Fact}

\newcommand{\norm}[1]{\| #1 \|}

\newcommand{\dmax}{d_{\mathrm{max}}}
\newcommand{\dmin}{d_{\mathrm{min}}}
\newcommand{\xstarmax}{x^*_{\mathrm{max}}}
\newcommand{\xstarmin}{x^*_{\mathrm{min}}}

\newcommand{\supp}{{\mathit{supp}}}
\newcommand{\abs}[1]{| #1 |}

\newcommand{\R}{\mathbb{R}}

\newcommand{\set}[2]{\{ \; #1\; | \; #2 \; \}}
\newcommand{\sset}[1]{\{ \; #1\; \}}

\newcommand{\argmin}{\mathop{\mathrm{argmin}}}
\newcommand{\assign}{\mathop{:=}}
\newcommand{\diag}{\mathrm{diag}}

\newcommand{\dist}{\mathrm{dist}}

\newcommand{\lref}[1]{(\ref{#1})}

\newcommand{\cN}{\mathcal{N}}
\newcommand{\cA}{\mathcal{A}}

\begin{document}

\title{Convergence of the Non-Uniform Directed Physarum Model}

\author{Enrico Facca\thanks{Department of Mathematics, University of Padova, Italy}
	\and Andreas Karrenbauer\footnote{Max Planck Institute for Informatics, Saarland Informatics Campus, Germany} \and Pavel Kolev$^\dagger$ \and Kurt Mehlhorn$^\dagger$}
\maketitle

\begin{abstract} 
	The directed Physarum dynamics is known to solve positive linear programs: 
	minimize $c^T x$ subject to $Ax = b$ and $x \ge 0$ for a positive  cost vector $c$. 
	The directed Physarum dynamics evolves a positive vector $x$ according to the
	dynamics $\dot{x} = q(x) - x$. Here $q(x)$ is the solution to $Af = b$ that 
	minimizes the ``energy'' $\sum_i c_i f_i^2/x_i$. 
	
	In this paper, we study the non-uniform directed dynamics $\dot{x} = D(q(x) - x)$, 
	where $D$ is a positive diagonal matrix.
	The non-uniform dynamics is more complex than the uniform dynamics 
	(with $D$ being the identity matrix),
	as it allows each component of $x$ to react with different speed 
	to the differences between $q(x)$ and $x$.
	Our contribution is to show that the non-uniform directed dynamics 
	solves positive linear programs.
\end{abstract}

\section{Introduction}

\emph{Physarum Polycephalum} is a slime mold that apparently is able to solve
shortest path problems. Nakagaki, Yamada, and T\'{o}th~\cite{Nakagaki-Yamada-Toth} report about 
the following experiment; see Figure~\ref{fig:maze}. They built a maze, covered it by pieces of
Physarum (the slime can be cut into pieces, which will reunite if brought into vicinity), 
and then fed the slime with oatmeal at two locations. After a few hours the slime retracted 
to a path following the shortest path in the maze connecting the food sources. 
The authors report that they repeated the experiment with different mazes; in all experiments, 
Physarum retracted to the shortest path. The paper~\cite{Tero-Kobayashi-Nakagaki} proposes 
a mathematical model, the \emph{Physarum dynamics}, for the behavior of the slime in the form 
of a system of coupled differential equations.
In~\cite{Physarum,Bonifaci-Physarum} it was shown that the Physarum dynamics solves 
the shortest path problem and the transportation problem.
It was soon asked whether the dynamics can also solve more complex problems.

A variant of the Physarum dynamics, the directed Physarum dynamics, is known to solve positive linear programs in standard form~\cite{Johannson-Zou,SV-LP}. A positive linear program asks to minimize a linear function $c^T x$ with a positive cost vector $c \in \R_{> 0}^m$ subject to the constraints $Ax = b$ and $x \ge 0$. Here $A \in \R^{n \times m}$ and $b \in \R^n$. Formally,
\begin{equation}\label{LP}
\text{minimize } c^T x \text{ subject to } Ax = b,\ x \ge 0. 
\end{equation}
We assume throughout that the system is feasible and use
\[ I = \set{i}{i \in \supp(x^*) \text{ for some optimal solution $x^*$ to }\lref{LP}} \] to denote the union of the supports of optimal solutions and
\[   F = \set{x}{\text{$x$ is a feasible solution to ~\lref{LP}}}\]
to denote the set of feasible solution. In order to avoid trivialities, we assume $I \not= \emptyset$. The \emph{directed Physarum dynamics} is defined as the dynamical system operating on $x \in \R_{> 0}^m$ according to 
\begin{equation}\label{uniform dynamics}
\dot{x} = q(x) - x,
\end{equation}
where $\dot{x}$ is the derivative of $x$ with respect to time, $x > 0$ and 
\[  q(x) = \argmin_f  \left\{\sum_{i} \frac{c_i}{x_i}f_i^2 \, | \, Af = b \right\} \]
is the minimum energy solution of $Af = b$ according to the weights (``resistances'') $c_i/x_i$. The system is initialized to a point $x^0 \in G \assign  \set{x}{x > 0} \subseteq \R^m$, i.e., $x(0) = x^0 \in G$.

\begin{theorem}[\cite{Johannson-Zou,SV-LP}] The dynamics~\lref{uniform dynamics} has a solution $x(t) \in G$ with $t \in [0,\infty)$. The solution satisfies
	\begin{enumerate}
		\item $\inf_t x_i(t) > 0$ for all $i \in I$,
		\item $\dist(x(t),F) \rightarrow 0$ as $t \rightarrow \infty$, and 
		\item $\lim_{t \rightarrow \infty} c^T x(t) = c^T x^*$.
	\end{enumerate}
\end{theorem}

In this paper, we study the \emph{non-uniform Physarum dynamics}
\begin{equation}\label{non uniform dynamics} 
	\dot{x} = D(q(x) - x),
\end{equation}
where for a fixed positive vector $d\in\R_{>0}^{m}$ we define by $D=\diag(d)\in\R^{m\times m}$ 
	a diagonal matrix with entries $D_{ij}=d_i$ if $i=j$ and $D_{ij}=0$ otherwise.
In this model \lref{non uniform dynamics}, different components of $x$ react with different speed 
to differences between $q(x)$ and $x$. Again the system is initialized with a point $x^0 \in G$, i.e., $x(0) = x^0$. 
The original and the non-uniform Physarum dynamics are both inspired by the study of the slime mold Physarum Polycephalum; 
see Section~\ref{Biological Background}. We generalize the theorem above to the non-uniform dynamics.

\begin{theorem} \label{Main Theorem} 
	The dynamics~\lref{non uniform dynamics} has a solution $x(t) \in G$ with $t \in [0,\infty)$. The solution satisfies
	\begin{enumerate}
		\item $\inf_t x_i(t) > 0$ for all $i \in I$,
		\item $\dist(x(t),F) \rightarrow 0$ as $t \rightarrow \infty$, and 
		\item $\lim_{t \rightarrow \infty} c^T x(t) = c^T x^*$.
	\end{enumerate}
\end{theorem}

This theorem was claimed in~\cite{Physarum-Complexity-Bounds}, albeit with an incorrect proof; see Section~\ref{Biological Background}. 
We note that the non-uniform dynamics~\lref{non uniform dynamics} is more complex that the uniform dynamics~\lref{uniform dynamics}. For example, if $x(0)$ is a feasible solution to~\lref{LP}, then the solution $x(t)$ is feasible\footnote{
	Consider $y(t) = Ax(t) - b$. Then $\dot{y} = A \dot{x} = A(q - x) = b - Ax = -y$. Thus
	$y_i(t) = y_i(0) e^{-t}$ for all $t$ and hence $y_i(0) = 0$ implies $y_i(t) = 0$ for all $t$. In particular, if $x(0)$ is feasible, $x(t)$ is feasible for all $t$. More generally, for arbitrary initial point, the residual $Ax(t) - b$ converges to zero exponentially fast.} 
for all $t$. In contrast, the non-uniform dynamics may move out of feasibility as 
Figure~\ref{example} illustrates.

\begin{figure}[t]
	\begin{center}
		\includegraphics[width=0.311\textwidth]{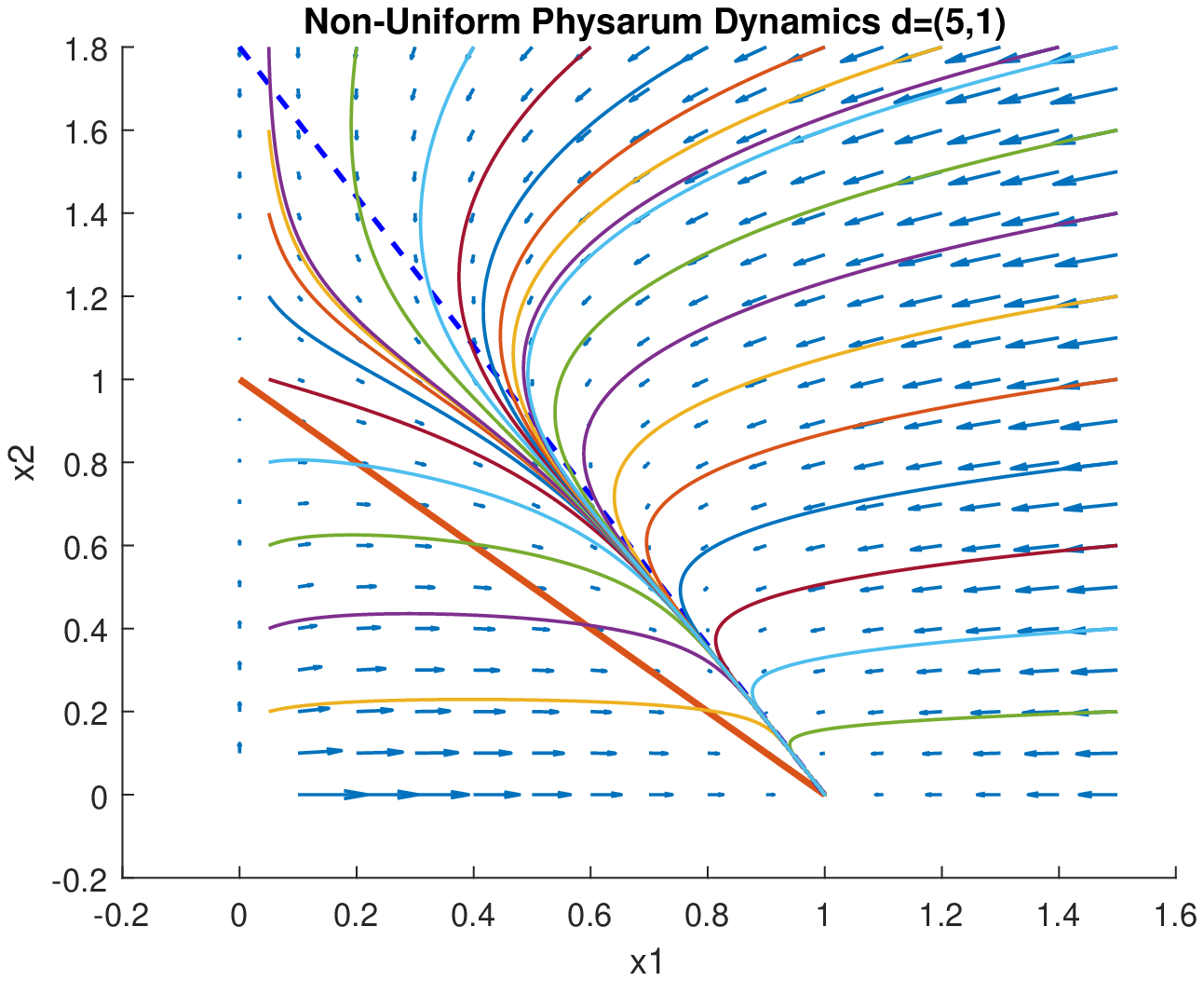}
		\hspace{0.4cm}\vspace{-0.1em}
		\includegraphics[width=0.311\textwidth]{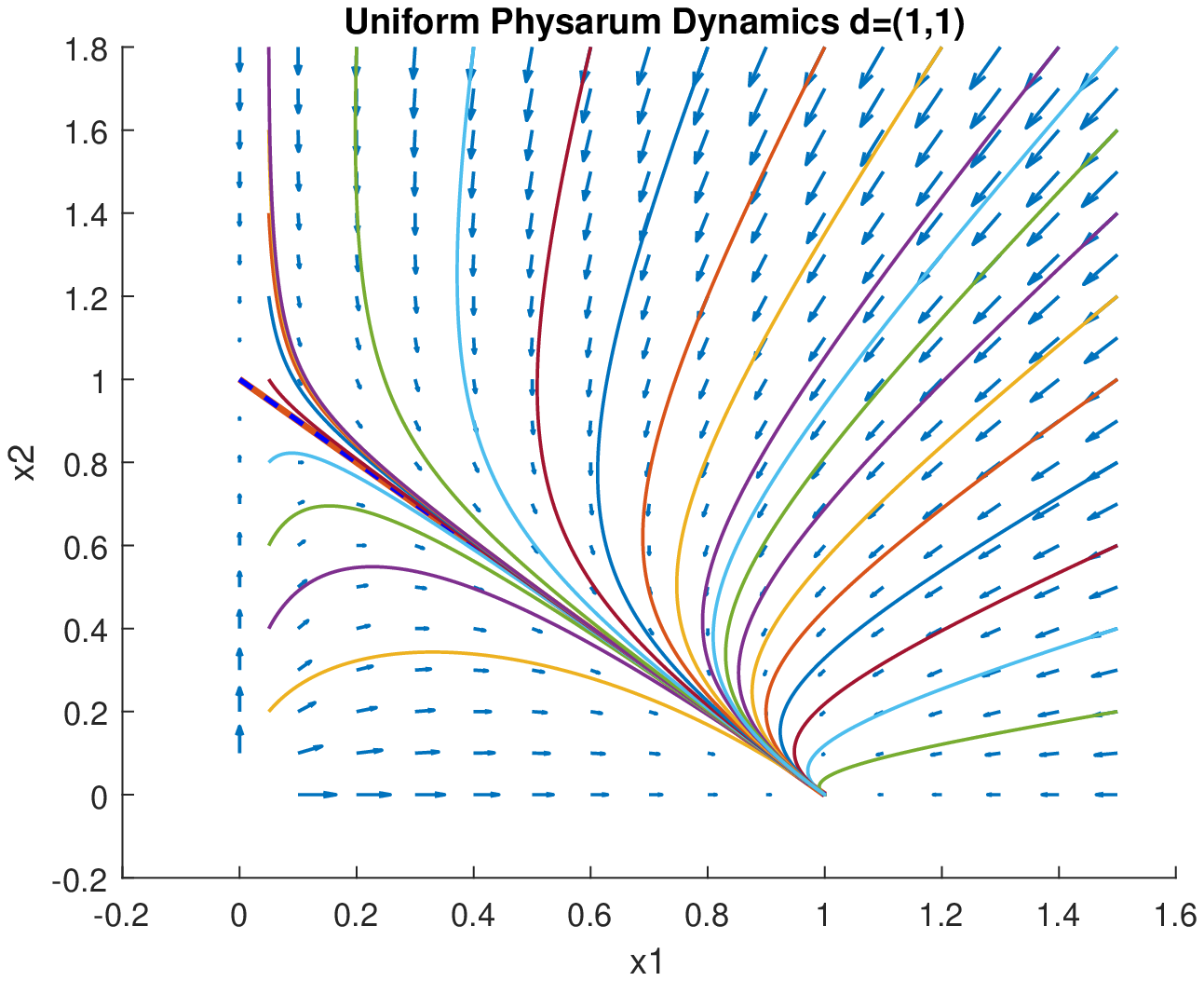}
		\hspace{0.4cm}\vspace{-0.1em}
		\includegraphics[width=0.311\textwidth]{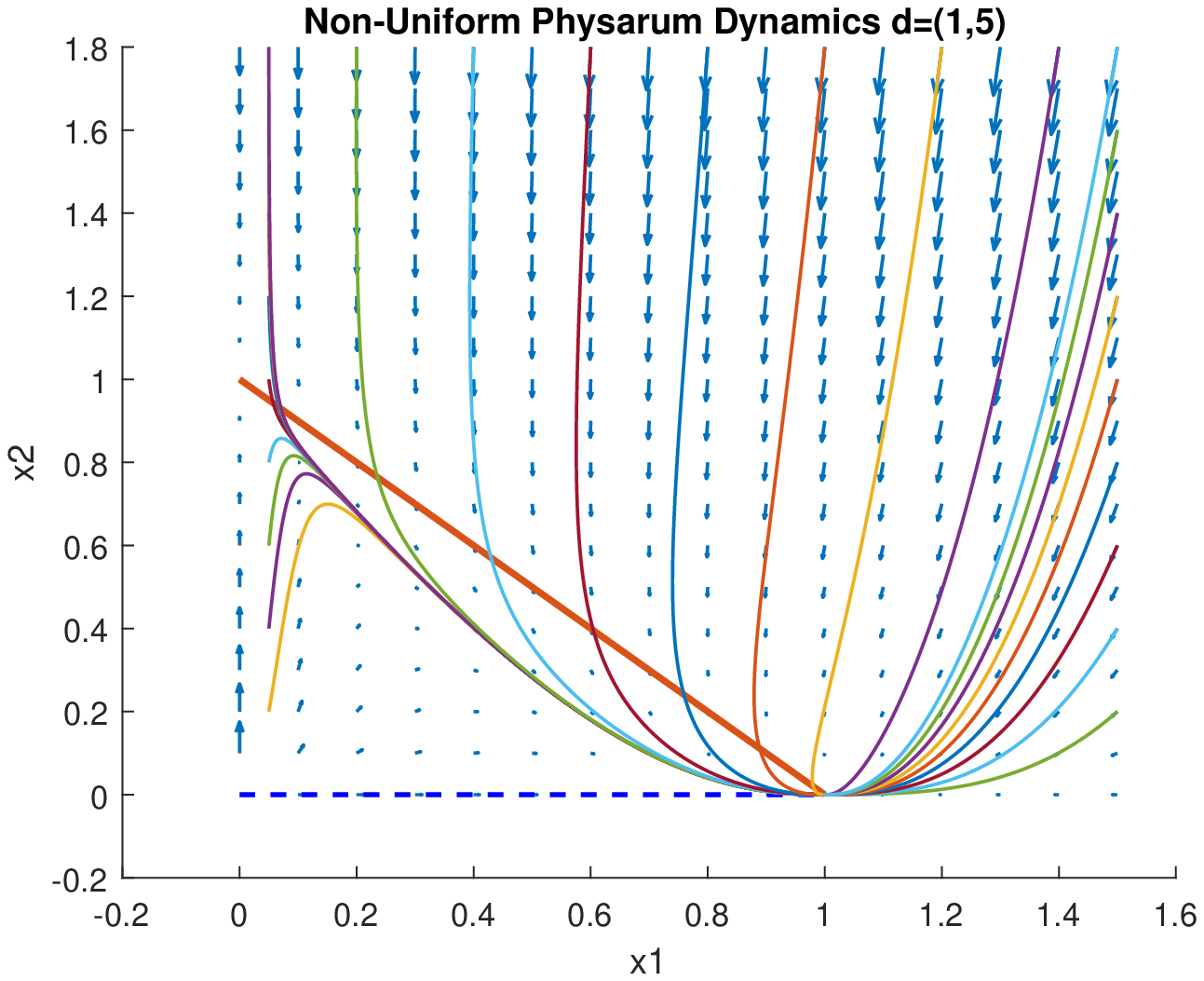}
	\end{center}
	\caption{\label{example}
		Consider the linear program 
		``$\text{minimize } x_1+2x_2 \text{ subject to } x_1+x_2 = 1 \text{ and } \ x \ge 0$''.
		The plot shows the flow fields $\dot{x}(t)$ of the corresponding 
		non-uniform dynamics $d=(5,1)$, the uniform dynamics $d=(1,1)$ and the non-uniform dynamics $d=(1,5)$, respectively. The feasible region ``$x_1 + x_2 = 1$, $x \ge 0$'' is shown in red. Also, several trajectories of $x(t)$ are depicted. For the uniform dynamics, the trajectories converge to the feasible region exponentially fast. In contrast, the solutions to the non-uniform dynamics converge to the feasible region only when they reach the optimal point $(1,0)$. Note that in the right plot ($d_1 = 1$ and $d_2 = 5$), the solutions enter the optimal point horizontally. Whether the solution enters from the left or from the right depends on the initial point. In the plot on the left ($d_1 = 5$, $d_2 = 1$), the solutions enter the optimal point with slope $-9/5$.
		The dashed blue line depicts this slope and an analytical solution is presented in Section~\ref{Approach to Optimum}.
	}
\end{figure}

A fundamental tool for proving convergence of a dynamical system is a Lyapunov function. It is a function defined on the states of the system whose derivative with respect to time is non-positive. We exhibit a Lyapunov function for the dynamics~\lref{non uniform dynamics}. Let $x^*$ be any optimal solution to~\lref{LP} with $I = \supp(x^*)$ (note that a convex combination of optimal solutions is also an optimal solution. Therefore there is an optimal solution with support equal to $I$) and consider

\begin{equation}\label{Lyapunov Function}
V(x) = 2 c^T D^{-1} x - \sum_i  \frac{c_i x_i^*}{d_i} \ln x_i
\end{equation}

\begin{theorem} $\dot{V}(x) \le 0$ for all $x \in G$ and there  is a constant $C$ (depending on $x(0)$) such that $V(x(t)) \ge C$ for all $t \in (0,\infty)$. \end{theorem}

It follows that $\lim_{t \rightarrow \infty} V(x(t))$ exists and $\lim_{t \rightarrow \infty} \dot{V}(x(t)) = 0$. By the first item of Theorem~\ref{Main Theorem} any limit point\footnote{A point $p$ is a (positive) limit point if there is a sequence $t_n \rightarrow \infty$ such that $x(t_n) \rightarrow p$ as $n \rightarrow \infty$.} of the dynamics~\lref{non uniform dynamics} will lie in
\[       G^* = \set{x}{x \ge 0 \text{ and } x_i > 0 \text{ for all } i \in I} .\]
It then follows that the dynamics converges to the set
\[                              E = \set{x \in G^*}{\dot{V}(x) = 0}.\]
We will see below that $\dot{V}(x) = 0$ implies $c^T x = c^T x^*$ and $q(x) = x$. In particular, $x \in E$ implies $x \ge 0$ and $x_i > 0$ for $i \in \supp(x^*)$, $q(x) = x$, i.e., $x$ is a fixed point of the dynamics and a feasible solution to the LP (since $Ax= Aq(x) = b)$, and $c^T x = c^T x^*$, i.e., $x$ is an optimal solution to the LP.

This paper is structured as follows. 
In Section~\ref{Biological Background} we review the biological background and discuss related work. 
In Section~\ref{Existence} we prove existence of a solution with domain $[0,\infty)$. 
In Section~\ref{Lyapunov} we analyze the Lyapunov function. 
In Section~\ref{Details of the Convergence Proof} we complete the proof of Theorem~\ref{Main Theorem}. 
In Section~\ref{Approach to Optimum}, we study in more detail how
the solution approaches the optimum in the example of Figure~\ref{example}.
In Section~\ref{sec:NonUnifMtxInit_NumEval}, we conduct a numerical experiment that explores 
the convergence-time dependence of the non-uniform dynamics~\lref{non uniform dynamics}
on the reactivity matrix $D$.

\section{Biological Background and Related Work}\label{Biological Background}

\begin{figure}[t]
	\begin{center}
		\includegraphics[width=0.4\textwidth]{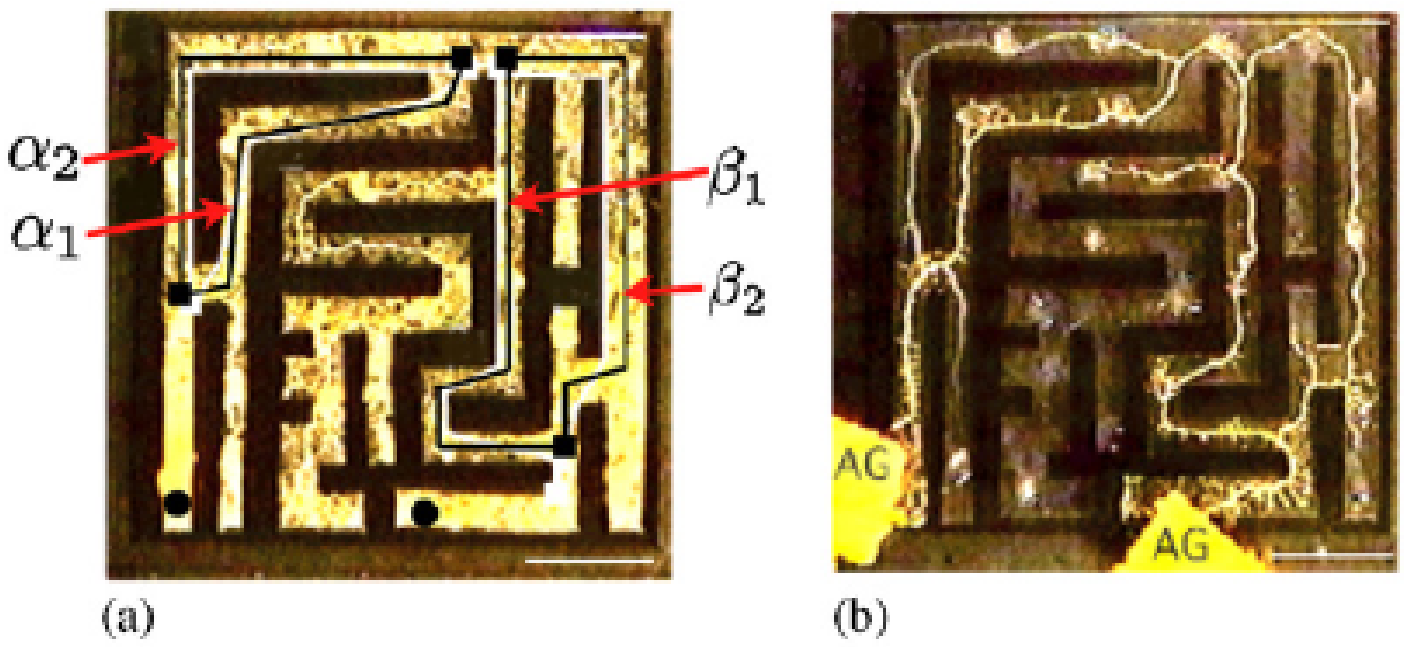}
		\hspace{0.3cm}\vspace{-0.1em}
		\includegraphics[width=0.37\textwidth]{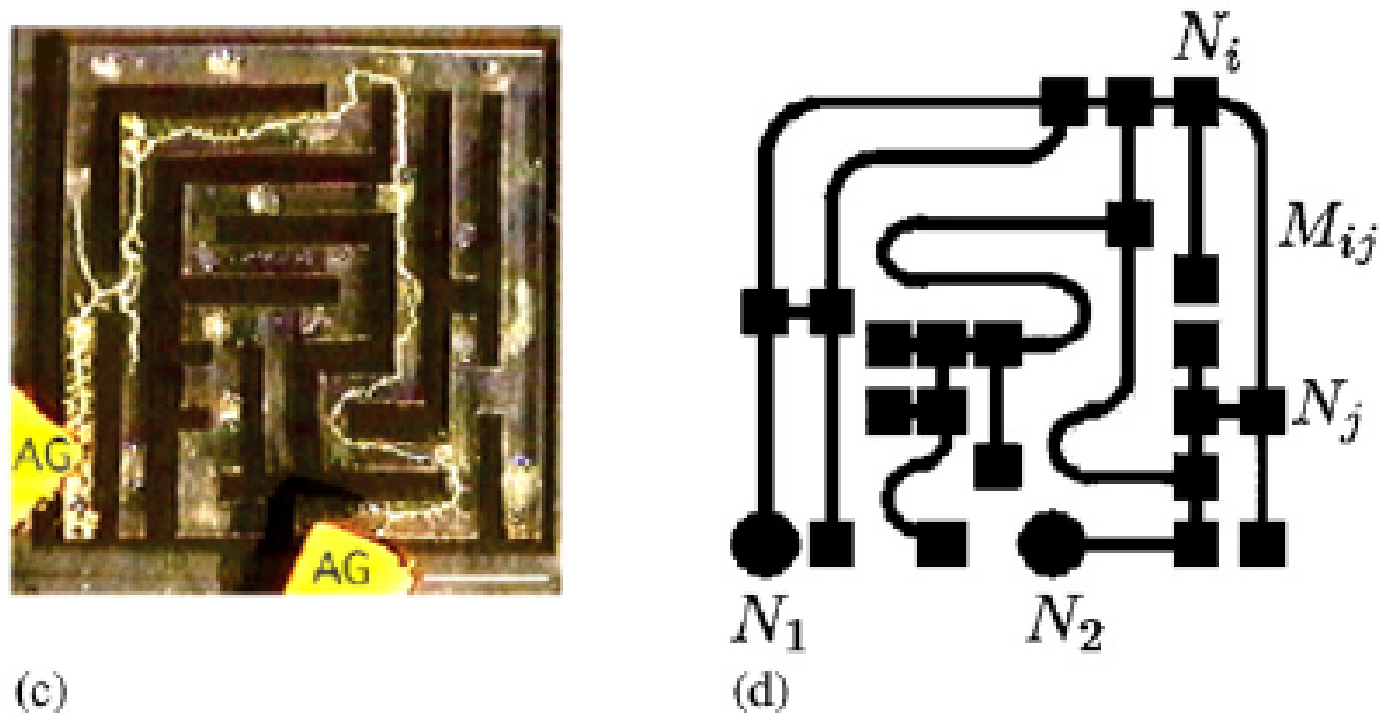}
	\end{center}
	\caption{\label{fig:maze} The experiment in~\cite{Nakagaki-Yamada-Toth}
		(reprinted from there): (a) shows the maze uniformly covered by Physarum;
		yellow color indicates presence of Physarum. Food (oatmeal) is provided at the
		locations labeled AG. After a while the mold retracts to the shortest path
		connecting the food sources as shown in (b) and (c). (d) shows the underlying
		abstract graph. The video~\cite{Physarum-Video} shows the
		experiment.
	}
\end{figure}

\emph{Physarum Polycephalum} is a slime mold that apparently is able to solve
shortest path problems. Nakagaki, Yamada, and T\'{o}th~\cite{Nakagaki-Yamada-Toth} report about the following experiment; see Figure~\ref{fig:maze}. They built a maze, covered it by pieces of Physarum (the slime can be cut into pieces which will reunite if brought into vicinity), and then fed the slime with oatmeal at two locations. After a few hours the slime retracted to a path following the shortest path in the maze connecting the food sources. The authors report that they repeated the experiment with different mazes; in all experiments, Physarum retracted to the shortest path.

The paper~\cite{Tero-Kobayashi-Nakagaki} proposes a mathematical model for the behavior of the slime and argues extensively that the model is adequate. Physarum is modeled as an electrical network with time varying resistors. We have a simple \emph{undirected} graph $G = (N,E)$ with two distinguished nodes modeling the food sources. Each edge $e \in E$ has a positive length $c_e$ and a positive capacity $x_e(t)$; $c_e$ is fixed, but $x_e(t)$ is a function of time. The resistance $r_e(t)$ of $e$ is $r_e(t) = c_e/x_e(t)$. In the electrical network defined by these resistances, a current of value 1 is forced from one of the distinguished nodes to the other. For an (arbitrarily oriented) edge $e = (u,v)$, let $q_e(t)$ be the resulting current over $e$. Then, the capacity of $e$ evolves according to the differential equation  
\begin{equation}\label{PD}
\dot{x}_e(t) = | q_e(t) | - x_e(t),
\end{equation}
where $\dot{x}_e$ is the derivative of $x_e$ with respect to time.

Nakagaki~et al.~\cite{PhysarumMinimumRiskPath} pointed out that different edges may react with different speed to the differences between flow and capacity. For example, Physarum prefers darkness over bright light and hence edges in a bright environment react differently than edges in darkness. This let to the non-uniform dynamics
\begin{equation}\label{NPD}
\dot{x}_e(t) = 
d_e(\abs{q_e(t)} - x_e(t)),
\end{equation}
where $d_e$ is an indicator for the reactivity of an edge.

The biological experiments concern shortest paths. 
The papers~\cite{Physarum,Bonifaci-Physarum} showed that the Physarum dynamics
can solve the shortest path problem and the transportation problem; 
here $A$ is the node-arc incidence matrix of a directed graph, 
$b$ is the supply-demand vector of a transportation problem, 
i.e., $\sum_i b_i = 0$, and $c > 0$ are the edge costs. 
It was soon asked whether the dynamics~\lref{PD} can also solve more complex problems:
\cite{SV-Flow} extended the result to more general flow problems,
\cite{ZhangM18} conducted an extensive experimental study on general flow problems
with convex and monotonically increasing objective functions,
and \cite{SV-IRLS,BeckerBonifaciKarrenbauerKolevMehlhorn,Facca-Cardin-Putti} 
showed that it can solve linear programs of the form
\begin{equation}\label{eq:absLP}
\text{minimize } c^T \abs{x} \text{ subject to } Ax = b,
\end{equation}
where $c$ is a positive vector\footnote{
	\cite{BeckerBonifaciKarrenbauerKolevMehlhorn} even shows that the dynamic \lref{PD}
	solves \lref{eq:absLP} when $c$ is a non-negative cost vector such that
	$c^T|v|>0$ for all $v\in\mathrm{ker}(A)$.}.
\cite{Karrenbauer-Kolev-Mehlhorn:NonUniformPhysarum} generalized the latter result
to the non-uniform dynamics~\lref{NPD}.
Moreover, the dynamics~\lref{NPD} inspired the Dynamic Monge-Kantorovich model which was introduced in~\cite{FCP18} and was recently applied in~\cite{FDCP20} for computing the numerical solution of the Optimal Transport Problem \cite{Ambrosio,Santambrogio}.

The directed version of the Physarum dynamics evolves according to the equation
\begin{equation}\label{directed dynamics}
\dot{x}_e(t) = q_e(t) - x_e(t).
\end{equation}
No biological significance is claimed for this dynamics.
\cite{Ito-Convergence-Physarum,SV-LP} showed that the dynamics~\lref{directed dynamics} 
	solves linear programs of the form
\begin{equation}\label{eq:dirLP}
\text{minimize } c^T x \text{ subject to } Ax = b, \ x \ge 0,
\end{equation}
where $c$ is a positive vector.
In~\cite{Physarum-Complexity-Bounds}, convergence was claimed for 
the non-uniform Physarum dynamics \lref{non uniform dynamics}, i.e.
\[ \dot{x}_e(t) = d_e (q_e(t) - x_e(t)). \]
Only a proof sketch was given; the claim has been withdrawn by the authors~\cite{erratum-Physarum-Complexity-Bounds}.
Here, we show convergence using a Lyapunov-based argument.

\section{Existence of a Solution}{\label{Existence}}

As mentioned above, we show that the dynamics~\lref{non uniform dynamics} with starting point 
$x(0) \in G$ has domain $[0,\infty)$ and stays in $G$. Moreover, for the limit points of the dynamics, the coordinates in $I$ are positive and the coordinates outside $I$ are zero. In order to talk conveniently about these limit points, we define the dynamics on a superset of $G$. Let
\begin{align*}         
G^* &= \set{x}{\text{$x \ge 0$ and $x_i > 0$ for $i \in I$}}
= \cup_{B;\ I \subseteq B \subseteq [m]} G_B^*,\\ \intertext{where}
G_B^* &= \set{x}{\text{$x_j > 0$ for $j \in B$ and $x_j = 0$ for $j \in [m] \setminus B$}}.
\end{align*}
We can define $q(x)$ for any $x \in G^*$: 
\[    
q(x) = \argmin_f \left\{ \sum_{i \in \supp(x)} \frac{c_i}{x_i} f_i^2 \, |\, \text{$Af = b$ and $\supp(f) \subseteq \supp(x)$} \right\}.
\]
be the minimum energy solution $q(x)$ with respect to the resistances $r_i = c_i/x_i$ and with support contained in $\supp(x)$. The following characterization of $q(x)$ is well-known. We use $R = R(t)$ to denote the diagonal matrix with diagonal entries $c_i/x_i$. 

\begin{lemma}\label{Formula for q} 
	For $x \in G^*$, let $B = \supp(x)$ and $N = [m] \setminus B$. Obtain $R_B$ by deleting the rows and columns corresponding to $N$, obtain $A_B$ from $A$ by first deleting the columns in $N$ and then keeping a maximum set of independent rows, let $b_B$ be the restriction of $b$ to these rows, and let $p_B$ be a vector of the same dimension as $b_B$. Then $q(x)$ is uniquely determined by $q = (q_B,q_N)$, $q_N = 0$, $q_B = R_B^{-1} A_B^T p_B$, where $p_B = (A_B R_B^{-1} A_B)^{-1} b_B$. Defining $p_h = 0$ for $h \not\in B$, we have the equalities 
	$ b^T p = q^T R q$.
\end{lemma}
\begin{proof} 
	For simplicity, we assume $\supp(x) = [m]$. By the KKT conditions for convex optimization, the optimal solution $q$ must satisfy $R q = A^T p$ for some $p$. Since $A q = b$ this implies, $A R^{-1} A^T p = b$ and $q = R^{-1} A^T p$. The matrix $A R^{-1} A^T$ is non-singular since $R$ is a diagonal matrix with positive entries and $A$ has full row rank. 
	Then, $b^T p = b^T_B p_B = (A_B q_B)^T p_B = q_B^T R_B R_B^{-1} A_B^T  p_B = q_B^T R_B q_B  = q^T R q$.
\end{proof}

The Physarum dynamics for $x \in G^*$ becomes
\begin{equation}    \dot{x}_i   = f_i(x) \assign \begin{cases}  d_i (q_i - x_i)    & \text{ if $i \in \supp(x)$ }\\
0                      &\text{ if $i \not\in \supp(x)$}
\end{cases}  \label{full non-uniform dynamics}
\end{equation}

We show existence of a solution with domain $[0,\infty)$ for each initial point $x^0 \in G^*$. Our argument is based on~\cite{SV-LP}.
Let $M = \max \sset{\text{absolute value of the determinant of a square submatrix of $A$}}$.
\begin{fact}[Lemma 3.3 in~\cite{BeckerBonifaciKarrenbauerKolevMehlhorn}]\label{bound on q} 
	For every $x \in G^*$: $\norm{q(x)}_\infty \le \beta \assign M\norm{b}_1$. 
\end{fact}

\begin{fact}[Lemma 5.2 in~\cite{SV-LP}]\label{fact:ATL1A} 	Let $x \ge 0$, $B = \supp(x) \not= \emptyset$, $L_B=A_BR_B^{-1}A_B^{T}$ and let $A_i$ be the $i$-th column of $A_B$. For every $i\in B$ it holds that
	\[
	\norm{A_B^{T}L_B^{-1}A_{i}}_{\infty}\leq\frac{c_{i}}{x_{i}}\cdot M
	\]
\end{fact}

\begin{lemma}\label{Extendability} 
	Suppose $x: [0,T) \mapsto G$ is a solution to~\lref{non uniform dynamics} for some $T > 0$. Let $B = \supp(x(0))$. Then
	\begin{compactenum}[\mbox{}\hspace{\parindent}a)]
		\item $x_i(t) \le \max(x_i(0),\beta)$ for $i \in [m]$ and $t \in [0,T)$.
		
		\item $x(T)=\lim_{t\rightarrow T^{-}}x(t)$ exists.
		
		\item $x_{i}(T)>0$ for every $i \in I$.
		
		\item $x_i(T)>0$ for every $i \in B$. 
	\end{compactenum}
\end{lemma}

\begin{proof} Let $i \in B$ be arbitrary. 
	\begin{compactenum}[a)]
		\item By Fact~\ref{bound on q}, $\dot{x}_i = d_i(q_i - x_i) \le d_i (\beta - x_i)$. Let $y_i(t) = (x_i(t) - \beta)/d_i$. Then $\dot{y}_i \le - d_i y_i(t)$ and hence $y_i(t) \le y_i(0) e^{-d_i t}$ by Gronwall's Lemma. Thus
		$$x_i(t) \le \beta + (x_i(0) - \beta) e^{-d_i t} \le \max(x_i(0),\beta).$$
		
		\item We know that $x_{i}$ is differentiable in the interval $[0,T)$ and
		$$\abs{\dot{x}_{i}(t)} = d_{i}\abs{q_{i}(t)-x_{i}(t)} \le  d_{i}\left(\abs{q_{i}(t)} + \abs{x_{i}(t)} \right)
		\le d_i \left( \beta + \max(x_i(0),\beta)) \right).$$
		Since every function that has bounded first derivatives is Lipschitz, 
		the limit $x(T)=\lim_{t\rightarrow T^{-}}x(t)$ exists.
		
		\item Let $y$ be an optimal solution with $\supp(y) = I$. Note that $I \subseteq \supp(x(0))$. Consider the barrier function 
		\[
		W(t)=\sum_{j \in I }\frac{c_{j}y_{j}}{d_{j}}\ln x_{j}(t).
		\]
		Since $x_j(t)$ is bounded by part a), so is $W(t)$. Suppose for a contradiction that $\inf_{t\in[0,T)}x_{i}(t)=0$ for some $i\in I$. Then, it follows that $\inf_{t\in[0,T)}W(t)=-\infty$, 
		implying that $\inf_{t \in [0,T)} \dot{W}(t)=-\infty$.
		However, since 
		\[
		\sum_{j\in I} y_{j}r_{j}(t)q_{j}(t)=y^{T}R(t)q(t)=y^{T}A^{T}p(t)=b^{T}p(t)=q(t)^{T}R(t)q(t)\geq0
		\]
		we obtain a contradiction to
		\[
		\dot{W}(t)=\sum_{j \in I}\frac{c_{j}y_{j}}{d_{j}}\cdot d_{j}\left[\frac{q_{j}(t)}{x_{j}(t)}-1\right]\geq-c^{T}y=-O(1).
		\]
		\item 
		Let $y$ be an optimal solution to~\lref{LP} with $\supp(y) = I$. By part c), we have $x_{i}(t)>0$ for every $i \in \supp(y)$ and every $t\in[0,T]$. Set 
		\[
		\epsilon=\min_{i\in \supp(y)}\min_{t\in[0,T]}\frac{x_{i}(t)}{y_{i}}
		\]
		and note that
		\[
		x(t)\geq\epsilon\cdot y,\quad\forall t\in[0,T]
		\]
		Further, using Fact \ref{fact:ATL1A}, observe that for every $t\in[0,T]$
		we have
		\begin{eqnarray*}
			\norm{A_B^{T}p_B(t)}_{\infty} & = & \norm{A_B^{T}L_B(t)^{-1}b_B}_{\infty}=\norm{ A_B^{T}L_B(t)^{-1}\sum_{i\in I} y_{i}A_{i}}_{\infty}\\
			& \leq & \sum_{i\in I} y_{i}\norm{A_B^{T}L_B(t)^{-1}A_{i}}_{\infty}\\
			& \leq & \sum_{i\in I}\frac{x_{i}(t)}{\epsilon}\cdot\frac{c_{i}}{x_{i}(t)}\cdot M \leq\norm{c}_{1}\cdot M /\epsilon.
		\end{eqnarray*}
		Let $N :=\lVert c\rVert_{1}\cdot M/\epsilon$ and $N':=\lVert d \rVert_{\infty}[\frac{N}{\min_{i}\{c_{i}\}}+1]$
		Then, for every $i\in B$ we have
		\[	
		\dot{x}_{i}(t)=d_{i}\left(q_{i}(t)-x_{i}(t)\right)=d_{i}x_{i}(t)\left(\frac{p_B^T[A_{B}]_i}{c_{i}}-1\right) 
		\geq -x_{i}\cdot\lVert d\rVert_{\infty}\left(\frac{N}{c_{i}}+1\right)\geq-x_{i}\cdot N'.
		\]
		Hence, by Gronwall's Lemma we obtain for every $t\in[0,T]$ and $i\in B$ that
		\[
		x_i(t)\geq e^{-t\cdot N'}\cdot x_i(0)>0.
		\]
	\end{compactenum}
	
\end{proof}

We use the following standard result from dynamic systems.
\begin{fact}
	\label{fact:DynSys} Let $G\subseteq R^{m}$ be an open set, $s\in G$
	and $f:G \rightarrow \R^{m}$ be a $C^1$ function (continuously differentiable function).
	Consider the following dynamical system: $\frac{d}{dt}x(t)=f(x(t))$
	with initial condition $x(0)=s$. Suppose it satisfies the following
	conditions for every solution $x:[0,T) \rightarrow G$ with $T \in (0,\infty)$:
	The limit $x(T) = \lim_{t\rightarrow T^{-}}x(t)$ exists and lies in $G$. 
	Then, there exists a global solution $x:[0,\infty) \rightarrow G$.
\end{fact}

We can now prove the existence of a global solution.

\begin{theorem}[\textbf{Existence of Solution}]
	\label{Thm: Existence} For every initial condition $x(0) \in G^*$,
	the non-uniform Physarum Dynamics (\ref{non uniform dynamics}) has a unique solution
	$x:[0,\infty)\mapsto G^*$. The solution satisfies $\supp(x(t)) = \supp(x(0))$ for all $t$. 
\end{theorem}
\begin{proof} 
	Let $B = \supp(x(0))$. There is a local solution $x(t)$ in an interval $[0,\epsilon)$ for some $\epsilon > 0$. The solution satisfies $x_i(t) > 0$ for $i \in B$ and $x_i(t) = 0$ for $i \not\in B$. The latter follows from the definition~\lref{full non-uniform dynamics} of the dynamics. On $G^*_B$ the function $x \mapsto D(q(x) - x)$ used in~\lref{full non-uniform dynamics} is defined by a vector of rational functions in $x$ and hence is continuously differentiable. Also $G_B^*$ restricted to the coordinates in $B$ is an open set. Thus Fact~\ref{fact:DynSys} is applicable and the existence of a global solution follows from Lemma~\ref{Extendability}. 
\end{proof}

\section{The Lyapunov Function}\label{Lyapunov}

Recall the definition of the Lyapunov function~\lref{Lyapunov Function}. Let $x^*$ be an optimal solution to~\lref{LP} with $I=\supp(x^*)$ and consider

\begin{equation}
V(x) = 2 c^T D^{-1} x - \sum_{i \in I}  \frac{c_i x_i^*}{d_i} \ln x_i.
\end{equation}

\begin{theorem}[\textbf{Lyapunov function}]\label{Thm: Lyapunov Function} 
	Let $x \in G^*$. Then
	\[ 
	\dot{V}(x) \le 0. 
	\]
	Moreover, $\dot{V}(x) = 0$ if and only if $x$ is an optimal solution to~\lref{LP}.
\end{theorem} 
\begin{proof} Since $V$ is continuously differentiable, $\dot{V}(x) = \nabla V \circ D (q(x) - x)$, where $\nabla V$ is the gradient of $V$ (vector of partial derivatives of $V$ with respect to the $x_i$'s)~\cite[page 30]{LaSalle}. An easy computation yields
	\begin{align}
	\frac{d}{dt}c^{T}D^{-1}x(t)&=c^{T}D^{-1}\cdot D\left[q(t)-x(t)\right]=c^{T}q(t)-c^{T}x(t).\label{eq:one}\\
	\intertext{and}
	\frac{d}{dt} \sum_{i\in I} \frac{c_i x_i^*}{d_i} \ln x_i(t)  & 
	=  \sum_{i\in I}c_{i}x_{i}^{*}\cdot\left(\frac{q_{i}(t)}{x_{i}(t)}-1\right)
	=  -c^{T}x^{*}+b^{T}p(t),\label{eq:two}
	\end{align}
	where the second equality follows by $r_{i}(t)=c_{i}/x_i(t)$ for $i\in I$, $q(t)=R^{-1}(t)A^{T}p(t)$,
	$Ax^{*}=b$ and 
	\[
	\sum_{i\in I}x_{i}^{*}r_{i}(t)q_{i}(t)=[x^{*}]^{T}R(t)q(t)=[x^{*}]^{T}A^{T}p(t)=b^{T}p(t).
	\]
	By combining (\ref{eq:one}) and (\ref{eq:two}) we obtain
	\begin{eqnarray}
	\dot{V}(x)
	= 2\left( c^{T}q(t)-c^{T}x(t)\right) +c^{T}x^{*}-b^{T}p(t).\label{eq:OBS}
	\end{eqnarray}
	Since $q(t)=R^{-1}(t)A^{T}p(t),$ it holds that 
	\[
	[q(t)]^{T}R(t)q(t)=[q(t)]^{T}A^{T}p(t)=b^{T}p(t),
	\]
	and using the Cauchy Schwarz inequality and the fact that $\sqrt{a} \cdot \sqrt{b} \le (a + b)/2$, we obtain
	\begin{eqnarray}
	c^{T}q(t) & = & [x(t)]^{T}R(t)q(t) = [R^{1/2} x(t)]^T [R^{1/2} q(t)]\nonumber \\
	& \stackrel{(1)}{\leq} & \sqrt{[x(t)]^{T}R(t)x(t)}\cdot\sqrt{[q(t)]^{T}R(t)q(t)}\nonumber \\
	& = & \sqrt{c^{T}x(t)}\cdot\sqrt{b^{T}p(t)}\nonumber \\
	& \stackrel{(2)}{\leq} & \frac{c^{T}x(t)+b^{T}p(t)}{2},\label{eq:cqt}
	\end{eqnarray}
	and further by combining (\ref{eq:OBS}) and (\ref{eq:cqt}),
	\begin{align*}
	\dot{V}(x) &=  2\left[c^{T}q(t)-c^{T}x(t)\right]+c^{T}x^{*}-b^{T}p(t)\\
	&\leq  b^{T}p(t)-c^{T}x(t)+c^{T}x^{*}-b^{T}p(t)\\
	& =  -c^{T}x(t)+c^{T}x^{*}\\
	& \stackrel{(3)}{\leq}  0.
	\end{align*}
	We have $\dot{V}(x) = 0$ if and only the inequalities (1), (2), and (3) are equalities. 
	(1) is an equality iff $q(x)$ and $x$ are parallel, i.e., $q(x) = \alpha\cdot x$ for some $\alpha\geq0$. (2) is tight iff 
	\[
	[x(t)]^{T}R(t)x(t)=[q(t)]^{T}R(t)q(t) = \alpha^2\cdot [x(t)]^{T}R(t)x(t),
	\] 
	i.e., iff $\alpha = 1$ and hence $q(x) = x$. (3) is tight iff $c^T x = c^T x^*$. 
	Finally, if $q(x) = x$ then $A x = A q(x) = b$. 
	Thus, $\dot{V}(x) = 0$ if and only if $x$ is an optimal solution to~\lref{LP}. 
\end{proof}

\begin{theorem}\label{no limit point on boundary of $G^*$} Let $x(0) \in G^*$. 
	Then there is an $\epsilon > 0$ depending on $x(0)$, $A$, $b$, $c$ such that 
	$x_h(t) > \epsilon$ for all $h \in I$ and all $t$. 
\end{theorem}
\begin{proof} 
	Let $x^*$ be an optimal solution to \lref{LP} with $\supp(x^*)=I$.
	Let $\xstarmax = \max_{i} x^*_i$,
	$\xstarmin = \min_{i \in I} x^*_i$, 
	$\dmax = \max_{i} d_i$, $\dmin = \min_{i} d_i$,
	$c_{\max}=\max_{i}c_i$, $c_{\min}=\min_{i}c_i$ and
	$\delta = \min(1,\min_{i\in I} x_i(0))$. 
	Since $\dot{V}(x) \le 0$ for all $t$, $V(x(t)) \le V(x(0))$ for all $t$. Then, for any $h \in I$ we have
	\[ 2 c^T D^{-1} x(t) + \frac{c_h}{d_h}\cdot x^*_h \cdot \log \frac{x_h(0)}{x_h(t)}  \le 2c^T D^{-1} x(0) + \sum_{e\in I\backslash\{h\}}\frac{c_{e}}{d_{e}}\cdot x_{e}^{*}\cdot\log \frac{x_{e}(t)}{x_{e}(0)}\]
	and hence 
	\begin{align*}
	\log\frac{1}{x_{h}(t)}&\leq \frac{d_{h}}{c_{h}}\cdot\frac{1}{x_{h}^{*}}\left(2c^{T}D^{-1}\left[x(0)-x(t)\right]+\frac{c_{h}}{d_{h}}\cdot x_{h}^{*}\cdot\log\frac{1}{x_{h}(0)}+\sum_{e\in I\backslash\{h\}}\frac{c_{e}}{d_{e}}\cdot x_{e}^{*}\cdot\log\frac{x_{e}(t)}{x_{e}(0)}\right)\\
	&\leq\frac{d_{\max}}{c_{\min}}\cdot\frac{1}{x_{\min}^{*}}\left(\frac{2mc_{\max}x_{\max}(0)}{d_{\min}}+\frac{c_{\max}}{d_{\min}}\beta\log\frac{1}{\delta}+\frac{(m-1)c_{\max}}{d_{\min}}\beta\log\left(\frac{\max\left\{ \beta,x_{\max}(0)\right\} }{\delta}\right)\right)\\
	&\leq 4m\cdot\frac{c_{\max}}{c_{\min}}\cdot\frac{d_{\max}}{d_{\min}}\cdot\frac{\max\left\{ \beta,x_{\max}(0)\right\} }{x_{\min}^{*}}\cdot\log\left(\frac{\max\left\{ \beta,x_{\max}(0)\right\} }{\delta}\right).
	\end{align*}
	for all $t$. The upper bound on $\ln (1/x_h(t))$ does not depend on
	$t$. This proves the claim.
\end{proof}

\begin{corollary}\label{positive limit points} 
	Assume $x(0) \in G^*$. The positive limit points
	of the trajectory $x(t)$ are contained in $G^*$. 
\end{corollary}

\begin{lemma}\label{f is continuous}
	The right hand side of~\lref{full non-uniform dynamics} is continuous. 
\end{lemma}
\begin{proof} 
	Let $(x^{(n)})$ be a sequence of points in $G^*$ converging to $x \in G^*$. 
	We need to show $f(x^{(n)}) \rightarrow f(x)$. For $i \in \supp(x)$ we must have $i \in \supp(x^{(n)})$ 
	for all sufficiently large $n$. Let $B$ be any superset of $\supp(x)$ and 
	consider the subsequence of $(x^{(n)})$ with $\supp(x^{(n)}) = B$. 
	We reuse $(x^{(n)})$ to denote the subsequence. 
	If the subsequence is finite, there is nothing to show. 
	So assume the subsequence is infinite. 
	We need to show $f_i(x^{(n)}) \rightarrow f_i(x)$. The proof proceeds by case distinction.
	
	\textbf{Case 1:} For $i \in \supp(x)$, we clearly have $f_i(x^{(n)}) \rightarrow f_i(x)$ 
	for $n \rightarrow \infty$, since $q_i$ is defined by a rational function 
	(Lemma~\ref{Formula for q}). 
	
	\textbf{Case 2:} For $i \in  B \setminus \supp(x)$, we need to prove $f_i(x^{(n)}) \rightarrow 0$.
	For this it suffices to show $q_i(x^{(n)}) \rightarrow 0$.
	Let $x^*$ be an optimal solution to~\lref{LP} with $\supp(x^*)=I$. 
	We may assume that $(x^{(n)})_i \ge x_i/2$ for all $i \in I$ and all $n$. Let 	
	\[
	\epsilon=\min_{i\in I} \frac{(x^{(n)})_{i}}{x^*_{i}} \ge \frac{1}{2}\cdot \min_{i\in I} \frac{x_{i}}{x^*_{i}}>0.
	\]
	Hence, $x^*\leq \epsilon\cdot(x^{(n)})$ and thus by Fact~\lref{fact:ATL1A} we have
	\[
	\norm{A_B^{T}p_B(t)}_{\infty}\leq\sum_{i \in I}x_i^*\norm{A_B^{T}L_B(t)^{-1}A_i}_{\infty}\leq\norm{c}_{1}\cdot M /\epsilon.
	\]
	By the proof of Lemma~\ref{Extendability} Part d) and by Lemma~\ref{Formula for q} it holds that
	\[ 
	\left|q_{i}(x^{(n)})\right|=\left|\frac{(x^{(n)})_{i}}{c_{i}}\cdot A_{B}^{T}p_{b}\right|\le\frac{(x^{(n)})_{i}}{c_{i}}\cdot\norm{c}_{1}\cdot\frac{M}{\epsilon}=(x^{(n)})_{i}\cdot O(1).
	\]
	Since $(x^{(n)})_i \rightarrow x_i=0$ for all $i\in [m]\backslash\supp(x)$, we have $\abs{q_{i}(x^{(n)})}\rightarrow 0$
	for all $i\in B\backslash\supp(x)$.
\end{proof}

\section{Putting it Together}\label{Details of the Convergence Proof}

We complete the proof of Theorem~\ref{Main Theorem}. Let $x(t)$, $t \in [0,\infty)$ be the global solution of \lref{full non-uniform dynamics} with $x(0) = x^0 \in G^*$. Then $x(t) \in G^*$ for all $t \in [0,\infty)$ by Theorem~\ref{Thm: Existence}. Suppose there is a sequence $t_n$ with $t_n \rightarrow \infty$ and $x(t_n) \rightarrow p$ as $n \rightarrow \infty$. Then $p \in G^*$ by Corollary~\ref{positive limit points}. Now $V(x(t_n)) \rightarrow V(p)$ and $\dot{V}(x(t_n)) \rightarrow \dot{V}(p)$. The latter equality holds since $V(x)$ is $C^1$ and hence $\dot{V}(x) = \nabla V \circ f(x)$, and since $f$ is continuous by Lemma~\ref{f is continuous}.

However, $\dot{V}(x(t)) \le 0$ and $V(x(t)) \ge V(p)$ for all $t \in [0,\infty)$ because $V(x(t))$ is decreasing. So $\dot{V}(x(t)) \rightarrow 0$ as $t \rightarrow \infty$. Hence, by uniqueness of limits $\dot{V}(p) = 0$.

Finally $\dot{V}(p) = 0$ implies that $p$ is an optimal solution to~\lref{LP} (Theorem~\ref{Thm: Lyapunov Function}).

\section{Approach to Optimum: A Simple Example}\label{Approach to Optimum}

We study the nonuniform dynamics for the simple example of Figure~\ref{example} in more detail. We consider 
``minimize $c_1x_1 + c_2x_2$, $x_1 + x_2 = 1$, $x \ge 0$, where $c_2 > c_1$'' 
and  the Physarum dynamics $\dot{x_i} = d_i (q_i - x_i)$, $i = 1,2$. The unique optimum is $(1,0)$. We ask how the dynamics enters the optimal point?

Assume, we are in the point $(x_1,x_2)$. The conductance of edge $i$ is $x_i/c_i$. The conductance of the system of parallel edges is $x_1/c_1 + x_2/c_2 = (x_1 c_2 + x_2 c_1)/c_1c_2$. Therefore the potential difference is
$\Delta = c_1c_2/(x_1 c_2 + x_2 c_1$ and hence with $N = x_1 c_2 + x_2 c_1$ 
\[
q_i = \frac{x_i}{c_i} \Delta = \frac{x_i c_{3 - i}}{x_1 c_2 + x_2 c_1} = \frac{x_i c_{3-i}}{N}
\]
and
\[
\dot{x}_i = d_i (q_i - x_i) = d_i x_i (c_{3 - i}/N - 1) = d_i x_i (c_{3 - i} - N)/N.
\]

We know that the dynamics converges to the point $(1,0)$ and are interested in the behavior near the limit point. Therefore let $(x_1,x_2) = (1 - \epsilon_1,\epsilon_2)$. Then
\begin{align*}
-\dot{\epsilon}_1 =  \dot{x}_1 &= d_1 (1 - \epsilon_1)  \frac{c_2 - c_2 (1 - \epsilon_1) - \epsilon_2 c_1}{c_2 - c_2 \epsilon_1 + \epsilon_2 c_1}  \approx d_1 \frac{c_2 \epsilon_1 - c_1 \epsilon_2}{c_2}\\
\dot{\epsilon}_2 =  \dot{x}_2 &= d_2 \epsilon_2 \frac{c_1 - c_2(1 - \epsilon_1) - c_2 \epsilon_2}{c_2 - c_2 \epsilon_1 + \epsilon_2 c_1} \approx d_2 \epsilon_2 \frac{c_1 - c_2}{c_2}.
\end{align*}
For the approximation, we only kept the terms linear in the epsilons and ignored all higher powers. This is justified since we are interested in the behavior for small $\epsilon_1$ and $\epsilon_2$. We continue with the approximations. 

From the second equation, we conclude
\[      \epsilon_2(t) = e^{- (c_2 - c_1)d_2/c_2 t}.\]
From the first equation, we conclude
\[   \dot{\epsilon}_1(t) = - d_1 \epsilon_1(t) + (d_1 c_1/c_2) \epsilon_2(t) \ge - d_1 \epsilon_1(t).\]
Thus $\epsilon_1(t) \ge e^{-d_1 t}$. If $d_1 < (c_2 - c_1) d_2 /c_2$ or equivalently $d_2 \ge d_1 c_2/(c_2 - c_1)$, the rate at which $\epsilon_2$ goes to zero is higher than the rate for $\epsilon_1$ and hence the trajectory converges to the $x$-axis and enters the optimal point horizontally.  

So assume $d_1 > (c_2 - c_1) d_2/c_1$. We try the ``Ansatz'' $\epsilon_1(t)  = e^{-d_1 t} f(t)$. Then $\dot{\epsilon}_1 = e^{-d_1t} \dot{f}(t) - d_1 e^{-d_1t} f(t)$ and hence
\[
	\dot{f}(t) = e^{d_1 t} (\dot{\epsilon}_1 + d_1 e^{-d_1 t} f(t))
	= e^{d_1 t}\left(-d_1 \epsilon_1(t) + \frac{d_1c_1}{c_2} \epsilon_2(t) + d_1 \epsilon_1\right)
	= \frac{d_1 c_1}{c_2} e^{\left(d_1 - (c_2 - c_1)\frac{d_2}{c_2}\right) t}.
\]
Thus
\[ 
f(t)=\frac{c_{1}d_{1}}{c_{2}d_{1}+\left(c_{1}-c_{2}\right)d_{2}}\cdot\left[e^{\left(d_{1}-d_{2}\cdot\frac{c_{2}-c_{1}}{c_{2}}\right)\cdot t}-1\right]+C
\]
and hence
\[ 
\epsilon_{1}(t)=\frac{c_{1}d_{1}}{c_{2}d_{1}-\left(c_{2}-c_{1}\right)d_{2}}\cdot\left[e^{-\left(d_{2}\cdot\frac{c_{2}-c_{1}}{c_{2}}\right)\cdot t}-e^{-d_{1}\cdot t}\right]+Ce^{-d_{1}\cdot t}.
\]
We conclude that $\epsilon_1$ and $\epsilon_2$ go to zero at the same rate, and the trajectory follows a straight line with slope 
\[
\tan\phi(t)=-\frac{\epsilon_{2}(t)}{\epsilon_{1}(t)}\approx \frac{\left(c_{2}-c_{1}\right)d_{2}-c_{2}d_{1}}{c_{1}d_{1}}.
\]
For $d_1 = (c_2 - c_1) d_2/c_2$, the analysis is inconclusive. However, since the behavior should be continuous in the $d$'s, it is natural to conjecture that the trajectories converge to the $x$-axis. 

Even less formally, this result can also be obtained as follows. Assume that the trajectory starting in the point $  (1 - \epsilon_1,\epsilon_2)$ is essentially straight, we must have (observe that the straight line from $  (1 - \epsilon_1,\epsilon_2)$  to the optimal point $(1,0)$ has direction $(\epsilon_1,-\epsilon_2)$)
\[    
\alpha \epsilon_1 = d_1 \frac{c_2 \epsilon_1 - c_1 \epsilon_2}{c_2}\quad\text{and}\quad - \alpha \epsilon_2 = d_2 \epsilon_2 \frac{c_1 - c_2}{c_2},
\]
for some $\alpha$. For $\epsilon_2 > 0$, we obtain $\alpha = d_2(c_2 - c_1)/c_2$ and further
$d_{1}c_{1}\cdot\epsilon_{2}=\left[d_{1}c_{2}-d_{2}\left(c_{2}-c_{1}\right)\right]\cdot\epsilon_{1}$
from the first equation. Thus,
\[   
\tan\phi(t)=-\frac{\epsilon_{2}(t)}{\epsilon_{1}(t)}=\frac{\left(c_{2}-c_{1}\right)d_{2}-c_{2}d_{1}}{d_{1}c_{1}},
\]
as above. Since $\epsilon_1 > 0$, we need  $d_{1}c_{2}>d_{2}(c_{2}-c_{1})$, i.e.,  $d_1 > (c_2 - c_1) d_2/c_2$. 

Let us specialize to the case $c_1 = 1$, $c_2 = 2$. 
\begin{itemize}
	\item If $d_2 > c_2 d_1 /(c_2 - c_1) = 2 d_1$, the trajectories enter the optimal point horizontally. Whether a trajectory enters from the left or from the right depends on the initial point. The right plot in Figure~\ref{example} shows an example. 
	\item If $d_2 < 2 d_1$, the trajectories enter the optimal point with slope 
	$-2+d_{2}/d_{1}$. In particular, for $d_1=d_2 = 1$ (uniform dynamics), the slope is $-1$. The left plot in Figure~\ref{example} shows an example. 
\end{itemize}

\section{Speed of Convergence}\label{sec:NonUnifMtxInit_NumEval}

At the suggestion of an anonymous reviewer,
we conduct a numerical experiment that explores the convergence-time
dependence of the non-uniform dynamics~\lref{non uniform dynamics}
on the reactivity matrix $D$.

Here, we consider a min-cost flow problem with a unit demand
between a single source-sink pair of nodes 
and parameterized costs $c^{(f)}$, see Figure~\ref{Graph}.
We compare the convergence-time of the dynamics~\lref{non uniform dynamics}
when the matrix $D$ is either:
i) $\diag(c^{(f)})$; or
ii) the identity matrix $I$ (i.e., the uniform dynamics~\lref{uniform dynamics}),
see Figure~\ref{ConvTime}.

\begin{figure}[H]
	\begin{center}
		\includegraphics[width=0.32\textwidth]{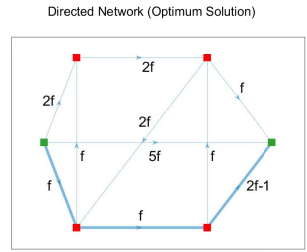}
	\end{center}
	\caption{\label{Graph}
		Consider a min-cost flow problem with a unit demand between a single source-sink pair of nodes.
		Let $N^{(f)}=(\cN,\cA,c^{(f)})$ be a directed network with a (green) source node on the left
		and a (green) sink node on the right.
		The optimum min-cost flow is unique, it has cost $4f-1$, and
		its arcs are made bold.}
\end{figure}

We implement and execute the forward Euler discretization of 
the non-uniform dynamics~\lref{non uniform dynamics}, which reads
\begin{equation}\label{eq:discrNUD}
x_e(t+1) = (1 - hd_e)\cdot x_e(t) + hd_e\cdot q_e(t),
\end{equation}
where the step size\footnote{Note that the step size satisfies $0<h\leq1/2$ 
	for reactivity matrices $D\in\{C,I\}$.} 
$h=\tfrac{1}{2\lVert c \rVert_{1}}$,
the number of iterations $t=\tfrac{1}{h}\log\tfrac{\lVert c \rVert_{1}}{\epsilon}$,
and the error $\epsilon=\tfrac{1}{10}$.

\begin{figure}[H]
	\begin{center}
		\includegraphics[width=0.315\textwidth]{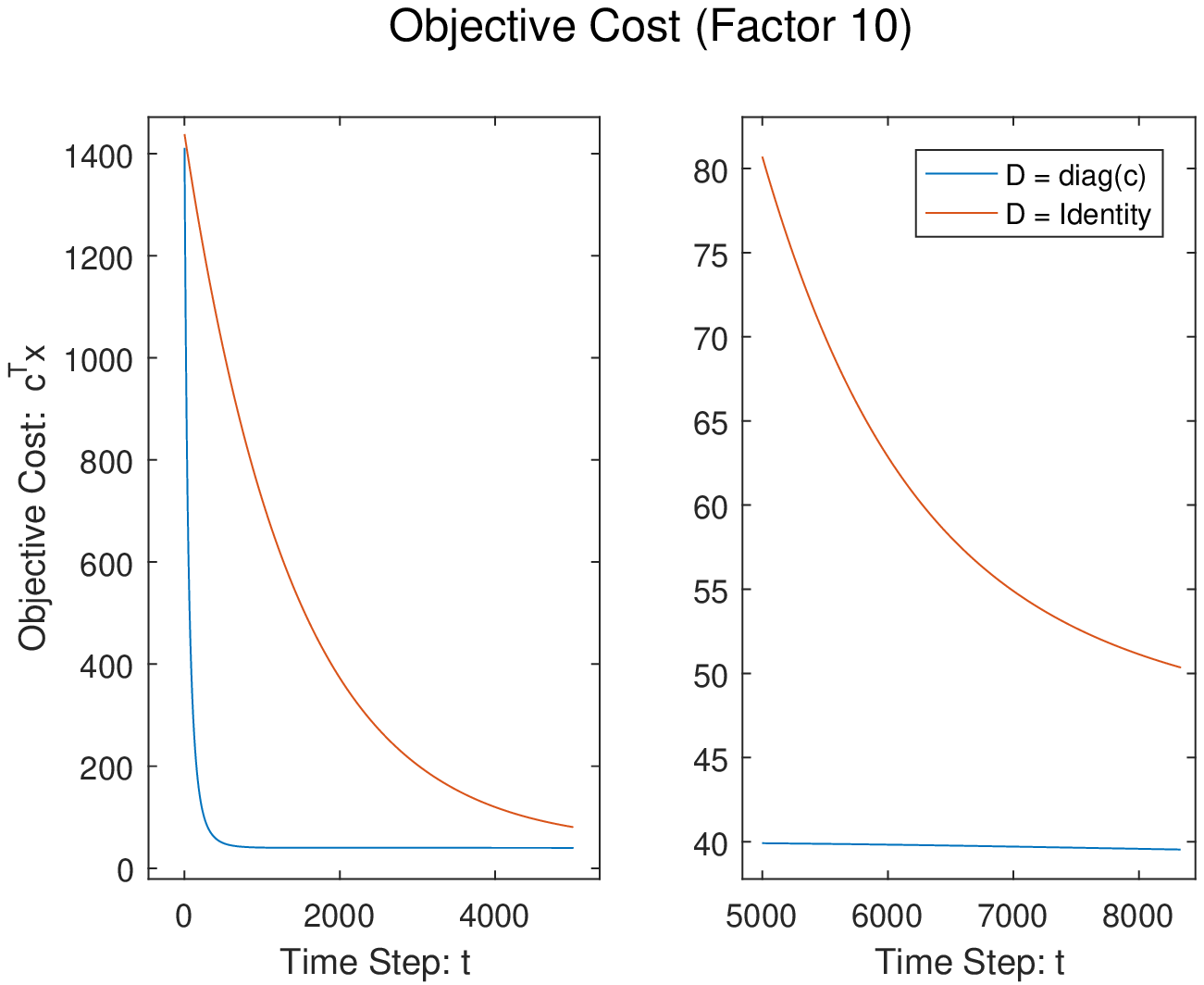}
		\hspace{0.4cm}\vspace{-0.1em}
		\includegraphics[width=0.315\textwidth]{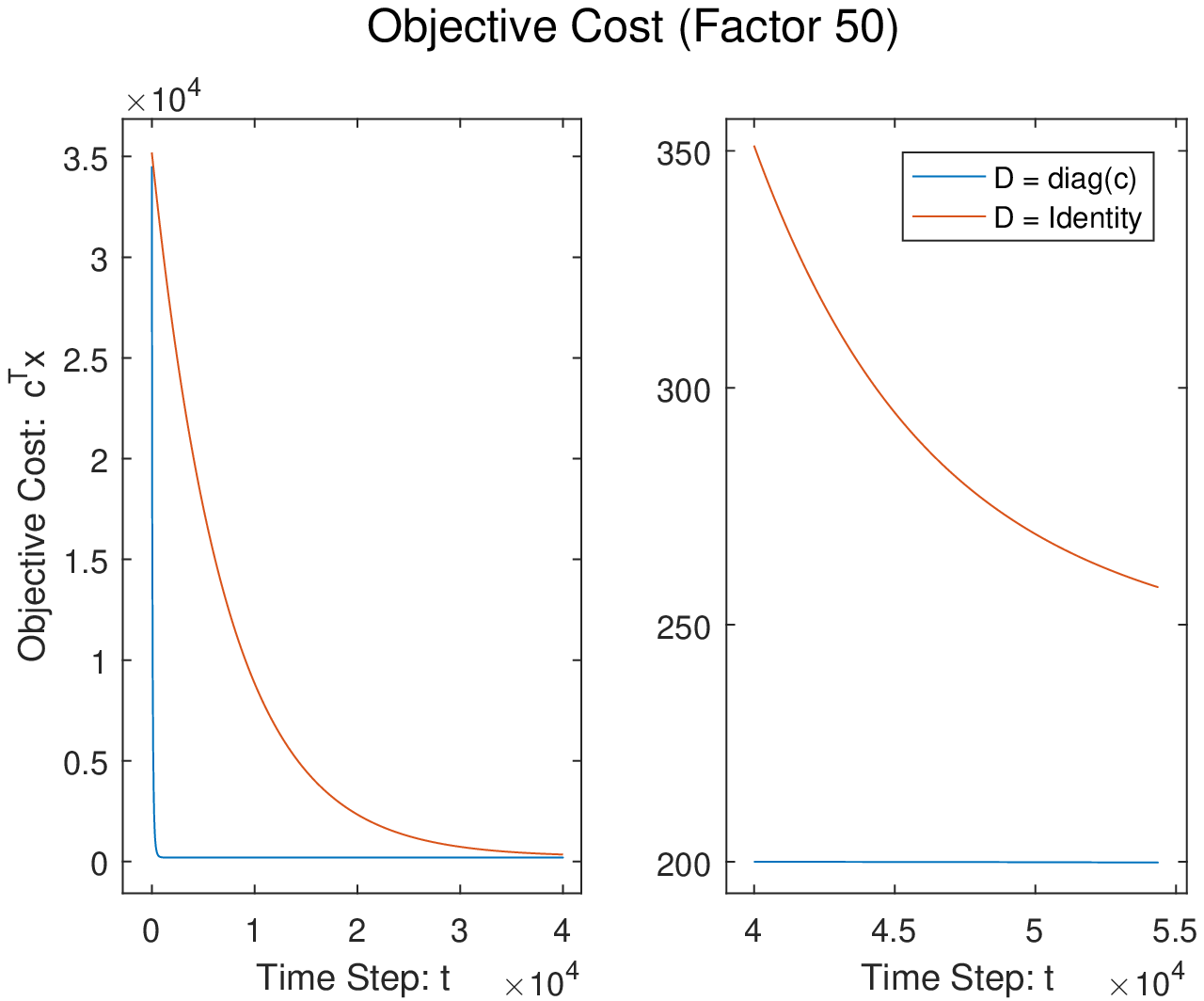}
		\vspace{0.4cm}\vspace{-0.1em}
		\includegraphics[width=0.315\textwidth]{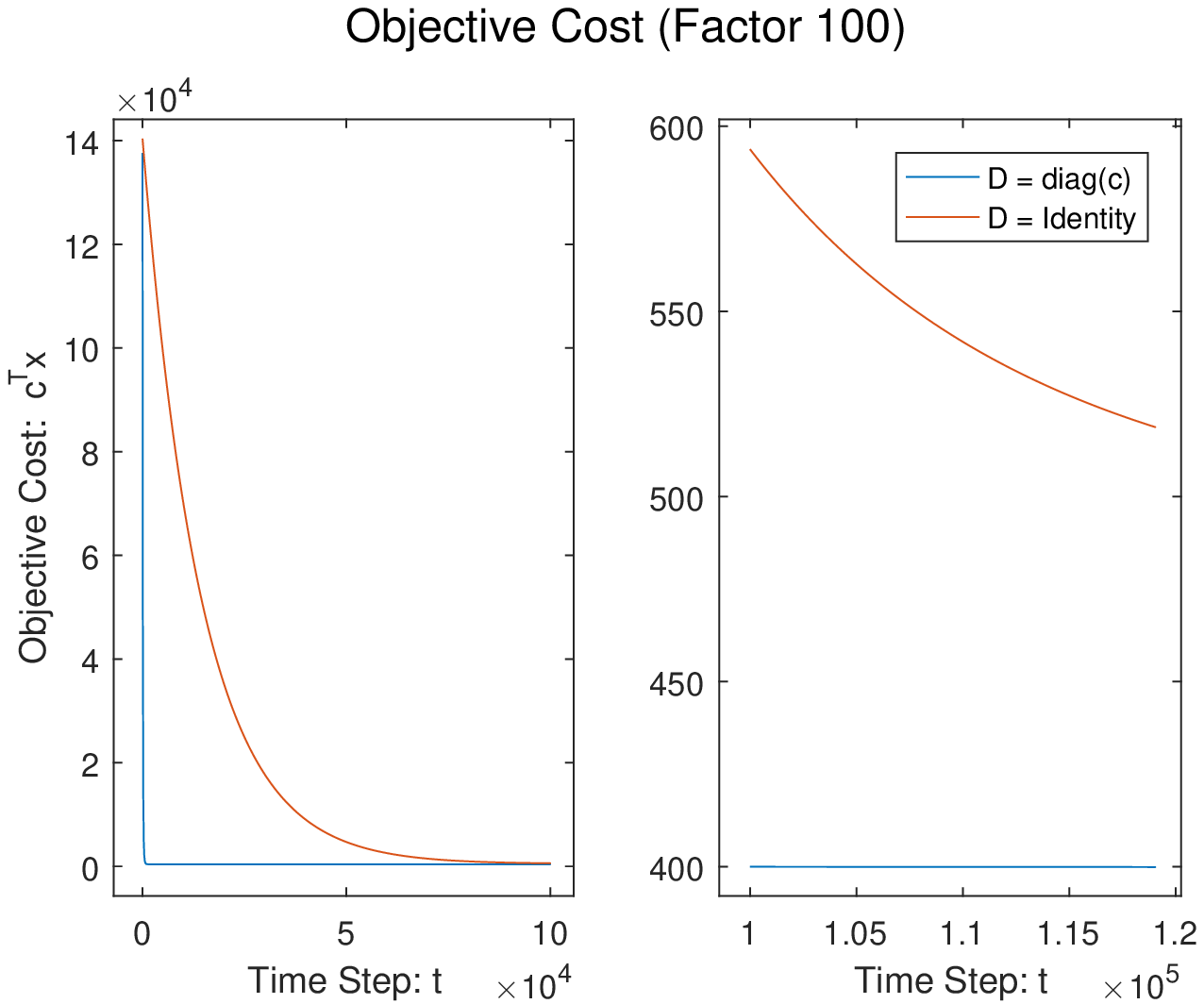}
	\end{center}
	\caption{\label{ConvTime}
		We consider the network from above for different $f\in\{10,50,100\}$.
		We initialize the discretized non-uniform dynamics~\lref{eq:discrNUD} with a fixed 
		vector $x^{(0)}$ such that $x_{e}^{(0)}=1/100$ for $e\in\cA_{opt}$ 
		and $x_{e}^{(0)}=100$ otherwise,
		and compare the convergence-time when matrix $D$ is either:
		i) $\diag(c^{(f)})$; or ii) the identity matrix $I$.}
\end{figure}

The experimental data in Figure~\ref{ConvTime} suggests that 
the non-uniform dynamics~\lref{non uniform dynamics} with $D=\diag(c^{(f)})$
can achieve a substantial convergence-time improvement
over the uniform dynamics~\lref{uniform dynamics}.

We leave the convergence-time dependence of the non-uniform dynamics~\lref{non uniform dynamics} 
on the reactivity matrix $D$ as a subject for future research.

\section{Acknowledgements}
KM wants to thank his colleague Mark Groves for a private lesson on dynamical systems.
The authors want to thank the anonymous reviewers for their careful reading,
constructive comments, and insightful suggestions.

\newpage
{\footnotesize
\bibliographystyle{alpha}
\bibliography{ref} }

\end{document}